\documentclass[a4paper,11pt]{amsart}
\usepackage{amsmath}  
\usepackage{amssymb}
\usepackage{amsthm}  
\usepackage{amscd}  
\usepackage{enumerate}
\usepackage{tikz}
\usepackage{dynkin-diagrams}
\usepackage{inputenc}
\theoremstyle{plain}
\newtheorem{thm}{Theorem}[section]

\newtheorem{lemma}[thm]{Lemma}
\newtheorem{cor}[thm]{Corollary}

\theoremstyle{definition}

\theoremstyle{remark}
\newtheorem{rem}{Remark}
\newtheorem*{exam}{Example}

\numberwithin{equation}{section}

\newcommand{\R}{\mathbf R}

\newcommand{\G}{\mathrm G}

\title[Holomorphic isometric embeddings]
{Holomorphic isometric embeddings of the complex two-plane Grassmannian into quadrics}

\author{Oscar MACIA, Yasuyuki NAGATOMO}

\address{Department of Mathematics, UNIVERSITY OF VALENCIA, 
C. Dr. Moliner, S/N, (46100) Burjassot, SPAIN} 
\email{oscar.macia@uv.es}

\address{Department of Mathematics, MEIJI UNIVERSITY, 
Higashi-Mita, Tama-ku, Kawasaki-shi, Kanagawa 214-8571, JAPAN} 
\email{yasunaga@meiji.ac.jp}



\begin{document}

\begin{abstract}
The present article studies holomorphic isometric embeddings
of the complex two--plane Grassmannnian into quadrics. 
We discuss the moduli space of these embeddings up to 
gauge and image equivalence using a generalisation of do Carmo--Wallach theory.
\end{abstract}
\maketitle

\section*{Introduction}

The description of the moduli of holomorphic isometric embeddings of K\"ahler manifolds into quadrics  can be  dealt with using gauge theory \cite{Na-13,Na-15}. The des-cription of the moduli in this sense is known, for instance, in the cases where the embedded manifold is $\mathbf C P^1$ \cite{MNT, MacNag} and, more recently in the general case of $\mathbf C  P^m$ \cite{Na-21}.  
The next  desideratum  is a study of embeddings of  general Grassmannians into quadrics. The aim of this article is to report the results of such an investigation.
In the present piece we examine the question for a restricted family of Grassmannians: the complex Grassmann manifolds 
${\rm Gr}_m (\mathbf C^{m+2})$ of all two--codimensional linear subspaces in $\mathbf C^{m+2}.$ This Hermitian
 symmetric space has a noteworthy geometrical structure as the
unique compact, K\"ahler, quaternionic K\"ahler manifold with positive scalar curvature \cite{Besse, Salamon, Wolf}.
The embedding of  arbitrary complex Grassmannians into quadrics will be analysed in a forthcoming paper \cite{next}.

 Before our goal can be achieved, and according to the general theory in \cite{Na-13},  it is necessary to investigate the representation theoretic description of certain modules associated to the relevant vector bundles appearing in the construction. In this sense, 
 the  central technical result proved in the paper is the determination of the decomposition of the space ${\rm GS}(\mathfrak mV_0,V_0)$ --to be defined in \S 1.5-- into irreducible representations in \S 3. From the plethysm  viewpoint, though Young tableaux and Littlewood--Richardson rules allow to systematically determine the decomposition of tensor products of  ${\rm SU}(n)$ representations \cite{Ful}, Littlewood did not laid  a straightforward method to determine which terms of the decomposition belong to the symmetric or skewsymmetric parts
 \cite{Littlewood}.  Nevertheless, we are able to overcome this
 difficulty by explicitly working out the action of the centre of the isotropy subgroup on the highest and lowest--weight vectors. 
The climax is achieved at the end of \S 3 where this result is used to describe the moduli up to gauge equivalence $\mathcal M_k.$
The special case of the complex two-plane Grassmannian in $\mathbf C^4,$ aka (complexified, compactified)  Minkowski space, being of particular interest, is worked out as an example.

For convenience of the reader, \S 1 contains a summary of the basic definitions and theorems
of the gauge theoretical approach to the study of moduli spaces of maps, that is, the generalisation of do Carmo--Wallach theory. This article being a sequel of sorts to \cite{MNT,MacNag} and \cite{Na-21}, this outline has been reduced to the bare minimum. 
It is nonetheless sufficient to understand the remaining sections of the paper without constant resource to these antecedents, nor to \cite{Na-13,Na-15} where  detailed proofs could be found if desired. 

Standard arguments in \S 2, where we prove the adequacy of the complex two-plane Grassmannian to the hypothesis of the generalisation of do Carmo--Wallach theory, have been improved 
and simplified  using 
Kodaira embeddings by
line bundles (\cite{ Na-15}, Theorem 6.1).  

\indent\paragraph{\bf Acknowledgements}
The work of O.M. is partially supported by  the AEI (Spain) and FEDER project  PID2019-105019GB-C21, and by the  GVA project AICO 2021 21/378.01/1.
The work of Y.N. is supported by JSPS KAKENHI Grant Number 26400074.

\section{Preliminaries}

\subsection{Evaluations and natural identifications.} 
Every manifold in this article is assummed to be connected.  The space of sections of the vector bundle $V\to M,$  will be denoted by  $\Gamma(V)$ and it determines  the   {\it evaluation map} $ev:M\times \Gamma(V)\to V : ev(x,t)=t(x).$ We do not make distinction between $ev$ and its restriction to   $M\times W$  where  $W$ is a subspace of $ \Gamma(V).$
Let the triple $(V\to M,h_V,\nabla^V)$ or $(V,h_V,\nabla^V)$ denote a vector bundle equipped with a fibre metric and a compatible connection. Two of these triples are said to be {\it isomorphic} if there exists a bundle map  between them preserving the fibre metrics and connections. Such a bundle map is called  a {\it bundle isomorphism}.
When considering holomorphic vector bundles over K\"ahler manifolds the previous description given by the triple $(V,h,\nabla)$
represents the Hermitian bundle $(V,h)$ equipped with the unique 
compatible Hermitian connection. In this context, the notion of
bundle isomorphism between triples $(V,h,\nabla)$ is equivalent to  {\it holomorphic
isomorphism} between Hermitian bundles.
The trivial vector bundle $M\times W\to M$ will be denoted by $\underline W$ or $\underline W\to M.$

Let $W$ be a real $N$-dimensional vector space
and ${\rm Gr}_p(W)$ the real Grassmannian of oriented $p$-planes in $W.$ Let $S\to {\rm Gr}_pW$  denote the \emph{tautological bundle.}
Then the \emph{universal quotient bundle} 
 $Q \to {\rm Gr}_p(W)$ 
is defined by the exactness of the sequence $0\to S \to \underline W\to Q\to 0$ of vector bundles  over ${\rm Gr}_p(W).$ 
The natural projection $\underline W \to Q$ allows to regard $W$ 
as a subspace of $\Gamma(Q)$ and is thus identified with the evaluation map.  Fixing an inner product on $W$ has the effect of inducing fibre--metrics and canonical connections on 
$S,\;Q\to {\rm Gr}_p(W).$

Let $f:M \to {\rm Gr}_p(W)$ be a map of a manifold $M$ into a Grassmannian, and denote by $f^{\ast}Q \to M$ the pull-back  of the universal quotient bundle. Then, $f$ is said to be \emph{full}
if the induced linear map 
$W\to \Gamma(f^{\ast}Q)$ is a monomorphism.

A real oriented vector bundle $V \to M$ of rank $q$ is {\it globally generated} by the $N$-dimensional subspace of sections $W\subset \Gamma(V)$ if the evaluation
map is surjective.
Moreover, if $W$ has an orientation, this defines a map $f:M\to {\rm Gr}_p(W),$ 
\[ 
f(x)= \text{Ker}\,ev_x=\left\{t\in W \,\vert \, t(x)=0 \right\},\qquad \dim \text{Ker}\,ev_x=p=N-q\]
where ${\rm Gr}_p(W)$ is a real oriented Grassmannian and the orientation
of $\text{Ker} \; ev_x$ is inherited from those of $V_x$ and $W.$ 
The map $f$ is said to be  {\it induced} by 
the pair $(V\to M,\;W).$ 
 The induced map determines a 
\emph{natural identification} of  $V\to M$  with $f^*Q \to M$ (cf.\ \cite{Na-13}),  and thus every smooth map $f:M\rightarrow {\rm Gr}_p(W)$ 
can be regarded as the induced map defined by  $(f^*Q\to M,\;W).$

If $W$ has an inner product, 
the adjoint of the evaluation map, $ev^{\ast}: V\to \underline{W},$ determines 
a natural identification  $V  \cong f^{\ast}Q$, as vector bundles over $M,$
when the image of $ev^{\ast}$ is restricted to $f^{\ast}Q \to M$. 
Therefore, $ev^{\ast}:V \to f^{\ast}Q$ is also termed a natural identification.

\subsection{image and gauge equivalence.}
Two mappings $f_1,f_2:M\to {\rm Gr}_p(W)$ are {\it image equivalent,} or {\it congruent},  
if there is an isometry $\phi$ of ${\rm Gr}_p(W)$ such that $f_2=\phi\circ f_1$. 
Each isometry $\phi$ of ${\rm Gr}_p(W)$ defines 
a bundle automorphism $\tilde \phi$ of $Q\to {\rm Gr}_p(W)$ covering the isometry $\phi$.  Besides,  if the  maps $f_1,f_2$ are 
 induced by $(V, W),$ 
they are {\it gauge equivalent}   
if they are image equivalent under $\phi$ and, additionally,  $ev_2^{\ast}=\tilde \phi\circ ev_1^{\ast}$, 
where $ev_i^{\ast}:V\to f_i^{\ast}Q$ $(i=1,2)$ are natural identifications.
Under these circumstances, the  connections on $V \to M$ induced by $ev_i^{\ast}$ 
are in the same orbit    
under gauge transformations. 
In \S 3 this will make possible to  identify the moduli space up to image equivalence $\mathbf M_k$ as a 
 quotient of the moduli space up to gauge equivalence $\mathcal M_k$  by the centraliser  
 of the holonomy subgroup of the structure group of $f^*Q.$

\subsection{Vector bundles over Hermitiann symmetric spaces.}
The complex two--plane Grassmannian is a compact Hermitian symmetric space associated to the Hermitian symmetric pair of compact type $(\mathrm G,\mathrm K)=({\rm SU}(N+2), \mathrm  S(\mathrm U(N)\times \mathrm U(2))$ or, in the diagrammatic notation of \cite{BastonEastwood},  
\[
\dynkin[%
scale=1.8] A{**...*X*}.
\]

In general terms,
if we regard $\mathrm G\to {\rm G/K}$ as a principal $\mathrm K$-bundle,  the canonical connection is  such that 
the horizontal subspace 
is given by the left translation of $\mathfrak m,$ where $\mathfrak m$ is defined by the  standard decomposition of Lie algebras 
$\mathfrak{g}=\mathfrak{k}\oplus \mathfrak{m}$, 
 $\mathfrak{g}$ and $\mathfrak{k}$ being the Lie algebras of $\mathrm G$ and $\mathrm K.$ 

Let $L\to \mathrm G/ \mathrm K$ be an associated, 
complex line bundle $ L=\mathrm G\times_{\mathrm K} V_0$,  
where $V_0$ is a complex $1$-dimensional 
$\mathrm K$-module.
The complex line bundle $L\to {\rm G/K}$ is actually holomorphic, by effect of  the canonical connection.
Moreover, there is  an invariant Hermitian metric $h$ on $L$  (indeed turning $L$ into an Einstein--Hermitian bundle), unique up to a positive constant multiple,  
for which the canonical connection is the Hermitian connection.
(cf,  \cite[p.121 Proposition 6.2]{Kob}). 

\subsection{Standard maps} If the space $W=H^0({\rm G/K},L)$ of holomorphic sections of $L \to {\rm G/K}$ is nontrivial, 
then, by the Bott--Borel--Weil Theorem,  it is an irreducible complex $\mathrm G$-module 
and globally generates  $L.$ 
Moreover, $W$ inherits a $\mathrm G$--invariant $L^2$--Hermitian inner product, from the Riemannian structure on ${\rm G/K}$ and the Hermitian structure of $(L,h).$ 
 The celebrated Kodaira  embedding of ${\rm G/K}$  is then the map induced by $(L\to {\rm G/K},\;W).$    The importance of the existence of a $\rm G$--invariant Hermitian inner product on $W$ in our theory can not be stressed enough, since it 
  allows us to obtain the pull--back connection. In fact, the pull--back metric induces a Hermitian--Yang--Mills connection.

The {\it real} vector bundle obtained from $L\to{\rm G/K}$ by restriction of the field of scalars to $\mathbf R$ is equipped with a fibre metric ${\rm Re}(h)$ and an
orientation induced by the complex structure. Furthermore,
 regard $W$ as a {\it real, oriented} vector space  with a $\rm G$--invariant $L^2$ inner product $(\cdot,\cdot)_{W}.$ The new induced  map into a real oriented Grassmannian is said to be {\it  standard}. 
 In this case, we  obtain a totally geodesic holomorphic embedding 
 of complex projective space into a real oriented Grassmannian,
and
the standard map is the composition of this last map with the Kodaira emedding.
Thus, the induced connection is also the Hermitian--Yang--Mills connection. 

\subsection{do Carmo--Wallach Theory}
Let $W$ be an orthogonal $\mathrm G$--module with respect to the inner product $(\cdot,\cdot)_W$. Regard $W$ as
a $\mathrm K$--module  where $\mathrm K$ is a subgroup of
$\mathrm G$ and let $V_0$ be a complex $\mathrm K$--submodule of $W,$  
following  \cite[Lemma 5.17]{Na-13}. The orthogonal splitting $W=U_0\oplus V_0$  allows the definition of a $\rm G$--equivariant standard map $f_0:M\to {\rm Gr}_p(W),$ with $p=\dim U_0,$ such that
$ f_0([g])=gU_0 \subset W,$ for all $ [g]\in {\rm G/K},\;g \in \mathrm G.$

Let $\mathrm{S}(W)$ denote  
the set of symmetric  endomorphisms of $W,$ 
equipped with the $\rm G$--invariant inner product
$(A,B)_{_S}=\text{trace}\,AB$, for $A,B \in \mathrm{S}(W).$
 Given $U,V$ subspaces of $W,$ define the real subspace
 $\mathrm S(U,V)\subset \mathrm S(W)$ spanned by 
 \[\mathrm S(u,v):=\frac{1}{2}\left\{u\otimes (\cdot, v)_{_W} + v\otimes (\cdot, u)_{_W} \right\},\]
where $u\in U,\;v\in V.$ 
Analogously,  $\G\mathrm{S}(U,V)\subset \mathrm{S}(W)$  is spanned by 
$g\mathrm{S}(u,v),$ with $g\in\mathrm{G}.$ 
Then, the  generalization
 of the do Carmo--Wallach Theorem  for holomorphic maps 
 of ${\rm G/K}$  into a quadric reads:
\begin{thm} \label{GenDW}

Let 
$f:{\rm G/K} \to {\rm Gr}_n(\mathbf R^{n+2})$ be a full holomorphic map
satisfying the {\it gauge condition for} $(L,h),$ that is,
the pull-back $f^{\ast}Q \to {\rm G/}K$ of the universal quotient bundle $Q\to {\rm Gr}_n(\mathbf R^{n+2})$ 
with the pull-back metric is holomorphically isomorphic to $(L,h)$. 
Let $W$ denote $H^0({\rm G/K},L)$ regarded as a real vector space equipped with a complex structure.
Then, there is a 
positive semi-definite symmetric endomorphism 
$T\in \text{\rm S}\,(W)$ 
such that the pair $(W,T)$ satisfies the following three conditions:\\
\begin{enumerate}[(a)]
\item The vector space $\mathbf{R}^{n+2}$ is a subspace of $W$ with the inclusion 
$\iota:\mathbf{R}^{n+2} \to W$ preserving the inner products, and 
$L\to M$ is globally generated by $\mathbf R^{n+2}$.
\item As a subspace, $\mathbf R^{n+2}=\text{\rm Ker}\,T^{\bot}$ and 
the restriction of $T$  is a positive symmetric endomorphism of $\mathbf R^{n+2}$. 
\item The endomorphism $T$ satisfies 
\begin{equation*}\label{HDW 2}
\left(T^2- I\!d_W, \G\mathrm{S}(V_0, V_0)\right)_{S}=0, 
\qquad
\left(T^2, \G\mathrm{S}(\mathfrak m V_0, V_0)\right)_{S}=0.
\end{equation*}
\end{enumerate}
\bigskip

If $\iota^{\ast}:W \to \mathbf R^{n+2}$ denotes the adjoint linear map of 
$\iota:\mathbf R^{n+2} \to W$, 
then $f:{\rm G/K} \to {\rm Gr}_n(\mathbf R^{n+2})$ is realized as 
 \begin{equation}\label{HDW5} 
f\left([g]\right)=\left(\iota^{\ast}T\iota \right)^{-1}\left(f_0\left([g]\right)\cap \text{\rm Ker}\,T^{\bot}\right), 
\end{equation}
where the orientation of 
$\left(\iota^{\ast}T\iota \right)^{-1}\left(f_0\left([g]\right)\cap \text{\rm Ker}\,T^{\bot}\right)_{[g]}$ 
is given by the orientation of $L_{[g]}$ and $\mathbf R^{n+2}$. 
Moreover, if the orientation of $\text{Ker}\,T$ is fixed, then 
we have a unique holomorphic totally geodesic embedding of 
${\rm Gr}_n(\mathbf R^{n+2})$ into ${\rm Gr}_{n^{\prime}}(W)$ 
by $\iota\left(\mathbf R^{n+2}\right)=\text{\rm Ker}\,T^{\bot}$, 
where $n^{\prime}=n+\text{\rm dim}\,\text{\rm Ker}\,T$
and a bundle isomorphism   
$\left(ev \circ \iota\circ\left(\iota^{\ast} T \iota\right)\right)^{\ast}
:L\to f^{\ast}Q$ as the natural identification by $f$.  
Such two maps $f_i$, $(i=1,2)$ are gauge equivalent if and only if 
$
\iota^{\ast}T_1\iota=\iota^{\ast}T_2\iota, 
$
where $T_i$ and $\iota$ correspond to $f_i$ in \eqref{HDW5}, respectively.\\ 

Conversely, 
suppose that a vector space $\mathbf R^{n+2}$ with an inner product and an orientation, and 
a positive semi-definite symmetric endomorphism 
$T\in \text{\rm End}\,(W)$ 
satisfying 
conditions (a),(b),(c) are given.  
Then we have a unique holomorphic totally geodesic embedding of 
${\rm Gr}_n(\mathbf R^{n+2})$ into ${\rm Gr}_{n^{\prime}}(W)$ after fixing the orientation of 
$\text{Ker}\,T$  
and 
the map $f:G/K \to {\rm Gr}_{p}(\mathbf R^{n+2})$ defined by \eqref{HDW5} 
is a full holomorphic map into ${\rm Gr}_n(\mathbf R^{n+2})$ 
satisfying the gauge condition with bundle isomorphism $L\cong f^{\ast}Q$ 
as the natural identification.

\end{thm}

\section{Characterisation of Holomorphic isometric embeddings of ${\rm Gr}_m(\mathbf C^{m+2})$ into quadrics}

The objective of this section is to determine if the complex two--plane Grassmannnian satisfies the hypothesis of  Theorem \ref{GenDW}, and how this special manifold fits with the general theory summarised in \S 1.

Note that, the real, oriented Grassmannian ${\rm Gr}_n(\mathbf R^{n+2}),$  can be regarded as a complex quadric hypersurface $\mathcal Q_n$ in $\mathbf C P^{n+1}$ (cf. [5, pp. 278--282]).
The universal quotient bundle over the quadric $\mathcal Q_n
$  has a holomorphic bundle structure 
induced by the canonical connection.
Notice that the curvature two-form $R_Q$ of the canonical connection 
on the universal quotient bundle $Q\to \mathcal Q_n$ is a constant multiple of the K\"ahler two-form
$\omega_{\mathcal Q}$ on $\mathcal Q_n$
\[R_Q=-{\sqrt{-1}}\omega_{\mathcal Q}.\]

In the case of the complex two--plane Grassmannian, the holomorphic line bundle $\mathcal O(1)\to {\rm Gr}_m\mathbf C^{m+2}$ is related to the universal quotient bundle $\tilde Q\to {\rm Gr}_m(\mathbf C^{m+2})$ by  $\mathcal O(1)\cong \wedge^2 \tilde Q.$ Therefore, the relation between the curvature two-forms $R_{\mathcal O(1)},\;R_{\tilde Q}$ of these bundles over the complex two-plane Grassmannian is ${\rm Tr}R_{\tilde Q}=R_{\mathcal O(1)},$ and,  again, the curvature is proportional to the K\"ahler form $\omega:$ $R_{\mathcal O(1)}=-\sqrt{-1} \omega.$

Let $f:\mathrm{Gr}_m(\mathbf C^{m+2}) \to \mathcal Q_n$ 
be a holomorphic embedding. 
Then $f$ is called  an {\it isometric embedding of degree $k$} if 
$f^{\ast}\omega_{\mathcal Q}=k \omega.$ Note that this implies that $k\in\mathbf N.$ 
By definition, 
$f: \mathrm{Gr}_m(\mathbf C^{m+2})\to\mathcal Q_n$ 
is a {\it holomorphic isometric embedding of degree $k$} iff
the pull-back of the canonical connection on the universal quotient bundle 
over $\mathcal Q_n$  
is  the canonical connection on $\mathcal O(k) \to  {\rm Gr}_m(\mathbf C^{m+2}),$ 
 namely, iff $f$ satisfies the gauge condition for $\left( \mathcal O(k), h_k \right),$ where $h_k$ denotes the standard Einstein--Hermitian metric on $\mathcal O(k)\to {\rm Gr}_m(\mathbf C^{m+2}).$
Since the canonical connection on $\left( \mathcal O(k), h_k \right)$ is  the Hermitian connection, 
 Theorem \ref{GenDW} allows us to determine the  moduli space $\mathcal M_k$ of holomorphic isometric embeddings of degree $k$ modulo  the gauge equivalence of maps.

Note that, by virtue of the identification $\mathcal Q_n={\rm Gr}_n(\mathbf R^{n+2}),$ switching simultaneously the orientation of $\mathbf R^{n+2}$ and of the subspace induces an {\it holomorphic} isometry $\tau$ of $\mathcal Q_n.$ In what follows, we do not distighuish maps $f:{\rm Gr}_m(\mathbf C^{m+2})\to \mathcal Q_n$ which differ by composition with $\tau.$

Let $(L,h)$ be a Hermitian holomorphic line bundle over a K\"ahler manifold  $M$.  Then, the complex structure $J_L$ of  $L\to M$
induces a complex structure  $J$ on $H^0(M,L)$   
such  that $J_L \circ ev=ev\circ J$, where 
$ev:\underline{H^0(M,L)} \to L$ is the evaluation map.  
If $W=\mathbf C^{l+1}\subset H^0(M,L)$
is a complex subspace of holomorphic sections,
it can also be regarded as 
 the real vector space $(\mathbf R^{n+2},J),$ where $n=2l,$ equipped with the complex structure $J.$ 
Let the inner product on $\mathbf R^{n+2}$ be defined  as the real part of the $L^2$--Hermitian inner product on $H^0(M,L).$
Consider the map $f:M\to \mathcal Q_n={\rm Gr}_n(\mathbf R^{n+2})$ induced by 
$(L\to M,\mathbf R^{n+2}).$ If $f$ 
is a holomorphic isometric immersion, 
because of $J_L\circ ev=ev\circ J$ it can also be regarded as a holomorphic isometric immersion into 
a totally geodesic $\mathbf CP^l\subset \mathcal Q_n.$ 

Denote by $\mathrm S(\mathbf R^{n+2})$ the space of
symmetric endomorphisms of $(\mathbf R^{n+2},J),$ and let
 $\mathrm H(\mathbf R^{n+2})\subset \mathrm S(\mathbf R^{n+2})$ be the
 subspace of Hermitian endomorphisms.
The  $\mathrm U(l+1)$--structure of $\mathbf R^{n+2}$
allows us to  define the orthogonal complement of $\mathrm H(\mathbf R^{n+2})\subset \mathrm S(\mathbf R^{n+2}).$
The orthogonal complementary  subspace $\mathrm H(\mathbf R^{n+2})^\perp=\left\{A\in \mathrm S(\mathbf R^{n+2})\,|\, JA=-AJ \right\} $ can be  identified with the $\mathrm U(l+1)$--module
$S^2\mathbf C^{l+1}.$   

It is well--known that the Kodaira embedding into a complex projective space
induced by $(\mathcal O(k)\to {\rm Gr}_m(\mathbf C^{m+2}),$ $\; H^0({\rm Gr}_m(\mathbf C^{m+2}),\mathcal O(k)))$  is rigid as a holomorphic isometric embedding: Calabi  \cite{Cal} proved that there is a unique class of congruent holomorphic isometric embeddings. 
When the bundle and space of sections 
are regarded as being real and oriented, the Kodaira embedding can
also be  viewed  as a holomorphic isometric embedding into a quadric.
Then, the characterisation of the moduli space of full holomorphic isometric embeddings ${\rm Gr}_m(\mathbf C^{m+2})\to \mathcal Q_n$ follows from Theorem 6.10 of \cite{Na-15}, which we recall here :

\begin{thm}
[{\rm \cite[Theorem 6.10]{Na-15}}]
\label{quadmodcpx}
Let $(L,h)$ be a Hermitian holomorphic line bundle over a K\"ahler manifold 
$M$. 
Let 
$\mathcal M$ be the moduli space of 
full holomorphic isometric immersions of 
$M$ into a quadric ${\rm Gr}_n(\mathbf R^{n+2})$ 
with the gauge condition for $(L,h)$ modulo gauge equivalence relation of maps.   

Suppose that  $\mathbf R^{n+2}$ is a {\it complex}\, subspace of $H^0(M,L)$ and 
the inner product on $\mathbf R^{n+2}$ is compatible with the complex structure. 
If there exists $f \in \mathcal M$ such that the evaluation map 
$ev:\underline{\mathbf R^{n+2}}\to L$ by $f$ satisfies $J_L\circ ev=ev\circ J$,  
then $\mathcal M$ has the induced complex structure and 
is an open submanifold of a complex  
subspace of $\mathrm  H(\mathbf R^{n+2})^\perp$. 
\end{thm}

Therefore, the moduli up to gauge equivalence of maps, of full holomorphic isometric embeddings 
${\rm Gr}_m(\mathbf C^{m+2})\to \mathcal Q_n$  satisfying the
 the gauge condition for 
$(\mathcal O(k) \to \mathbf CP^m,h_k)$ 
 is an open submanifold of a complex  
subspace of $\mathrm  H(\mathbf R^{n+2})^\perp$, 
where $\mathbf R^{n+2}=H^0\left({\rm Gr}_m(\mathbf C^{m+2}), \mathcal O(k)\right)$.

\section{Moduli spaces}

The gauge theoretic generalization of the do Carmo--Wallach theory requires a thorough understanding of the  symmetric endomorphisms of the space $W=H^0({\rm G/K},L)$  of holomorphic sections of certain line bundles over ${\rm G/K}.$

Consider the symmetric pair 
$({\rm G,K})=\left(\text{SU}(m+2), \text{S}\left(\text{U}(m)\times \text{U}(2)\right)\right)$, 
where an element in the isotropy subgroup 
${\rm K}=\text{S}\left(\text{U}(m)\times \text{U}(2)\right)$ is of the form:
\[
\begin{pmatrix}
A & O \\
O & B
\end{pmatrix}, \quad A \in \text{U}(m), \,\,B \in \text{U}(2),\qquad |A||B|=1.
\] This corresponds to the description of the complex two--plane Grassmannian as a Hermitian symmetric space ${\rm Gr}_m(\mathbf C^{m+2})=\frac{{\rm SU}(m+1)}{{\rm S}({\rm U}(m)\times {\rm U}(2))}.$

Several $\mathrm U(n)$ representations will be used in this section. Let us fix a maximal torus inside $\mathrm U(n)$ defined by the set of diagonal matrices, which will be denoted by ${\rm diag}(x_1,\dots,x_n),$ and denote by $V_n(\lambda_1,\lambda_2, \cdots,\lambda_n)$  
an irreducible complex representation of 
$\text{U}(n)$ with the highest weight 
$\lambda=(\lambda_1,\lambda_2, \cdots,\lambda_n)$, 
where $\lambda_1 \geqq \lambda_2 \geqq \cdots \geqq \lambda_n$ are integers.  Then, the maximal torus acts on the highest--weight vector $\hat w\in V_n(\lambda)$  as
\[{\rm diag}(x_1,\dots,x_n)\hat w = x_1^{\lambda_1}\dots x_n^{\lambda_n}\hat w.\]
 We can also regard $V_n(\lambda)$ as a ${\rm SU}(n)$ representation, and in this case it will  satisfy $V_n(\lambda_1,\dots,\lambda_n)=V_n(\lambda_1+t,\dots,\lambda_n+t),\; t\in\mathbf Z,$ so by convention we will
choose $t=-\lambda_n$ so the last weight becomes zero.
The {\it basis of dominant integral weights} of ${\rm SU}(n)$ 
will be denoted by $\pi_i\; (1\leq i \leq n-1)$ and are defined by
$\lambda_1=\dots=\lambda_i=1,\; \lambda_{i+1}=\dots=\lambda_n=0.$ The ${\rm SU}(n)$ irreducible representation with highest weight $\pi=\sum_i k_i\pi_i$ will be denoted by $F_n(\pi).$ 
 In diagrammatic notation, the fundamental representations $F_n(\pi_i)$ are given by  labeled Dynkin diagrams
\[F_n(\pi_i)=
\dynkin[%
labels*={,,1,,},
labels={,, i^{th},},
scale=1.8]A{**...**...*}.
\]

The relation between ${\rm SU}(n)$ representations $F_n(\pi)$ 
 and $V_n(\lambda)$ is
\[
F_n\left(\sum_{i=1}^{n-1} k_i\pi_i\right)
=V_n\left(\sum_{i=1}^{n-1} k_i, \sum_{i=2}^{n-1} k_i, \cdots, k_{m-2}+k_{n-1},k_{n-1},0\right). 
\]

The subindex $n$ in $F_n,\;V_n$ will be neglected when it is clear enough. Let $\pi=\sum_i k_i\pi_i$ a dominant integral weight, associated to the ${\rm SU}(m+2)$ representation $F(\pi).$ An element of the maximal torus of $\mathrm U(m+2)$ acts on the highest  weight vector $\hat w\in F(\pi)$ as 
${\rm diag}(x_1,\dots,x_{m+2})\hat w=x_1^{\lambda_1}\dots x_{m+1}^{\lambda_{m+1}}\hat w.$ Since the dual representation of $V(\lambda)$ has lowest weight equal to $-\lambda,$ the torus of ${\rm U}(m+2)$ acts on the lowest weight vector $\check w\in F(\pi)$ as it does on the highest weight vector of $F(-\pi ')$ with
\[
\pi'=-\sum_{i=1}^{m+1} k_i\pi_{m+2-i}=-\sum_{j=1}^{m+1} k_{m+2-j}\pi_j. 
\] 
Thus we obtain 
\begin{eqnarray*}
{\rm diag}(x_1,\dots ,x_{m+2}) \check w 
&=&
x_1^{-\sum_{j=1}^{m+1} k_{m+2-j}} 
x_2^{-\sum_{j=2}^{m+1} k_{m+2-j}}\cdots 
x_{m}^{-k_{2}-k_{1}} 
x_{m+1}^{-k_{1}} \check w\\
&=&
 x_1^{-\sum_{j=1}^{m+1} k_{j}} 
x_2^{-\sum_{j=1}^{m} k_{j}} \cdots  
x_{m}^{-k_{2}-k_{1}} 
x_{m+1}^{-k_{1}}\check w.
\end{eqnarray*}

We take an element of the center ${\rm U}(1)$ of ${\rm K}=\mathrm  S(\mathrm U(m)\times \mathrm U(2))={\rm SU}(m)\times {\rm SU}(2)\times {\rm U}(1)$: 
\[
Y=\text{diag}
(\;
\underbrace{y^{\frac{1}{m}},\dots,y^{\frac{1}{m}}}_{m}, 
y^{-\frac{1}{2}}, y^{-\frac{1}{2}}\; ). 
\] 
Consequently, $Y$ acts on $\check w$ as 
\begin{align*}
Y\cdot w
=&\underbrace{y^{-\frac{1}{m}\sum_{j=1}^{m+1} k_{j}}
y^{-\frac{1}{m}\sum_{j=1}^{m} k_{j}}\cdots 
y^{-\frac{1}{m}(k_{3}+k_{2}+k_{1})}
y^{-\frac{1}{m}(k_{2}+k_{1})}}_{m}
y^{\frac{1}{2}k_{1}}\check w\\
=&y^{-\frac{mk_1}{m}-\frac{mk_2}{m}
-\frac{(m-1)k_3}{m}-\cdots -\frac{2k_{m}}{m}-\frac{k_{m+1}}{m}+\frac{1}{2}k_1}\check w\\
=&y^{-\frac{1}{2}k_1-\frac{1}{m}\sum_{i=2}^m (m+2-i)k_i}\check w.
\end{align*}

Since $H^0(M;\mathcal O(k))=F(k \pi_2)$ 
by the Bott--Borel--Weil Theorem \cite{Bott}, 
the moduli space is contained in 
$S^2(F(k\pi_2))$.  This last space was denoted by $\mathrm  H(W)^\perp $ in \S 2.
Firstly, by  steady application of  Littlewood--Richardson rules on Young tableaux  with two rows and $k$ columns \cite{Ful},
  we have that 
\begin{lemma}
\begin{eqnarray*}
F(k\pi_2) 
\otimes F(k\pi_2) & = & 
\bigoplus_{i=0}^{k}\bigoplus_{j=i}^{k}V(2k-i,2k-j,j,i,0,\cdots, 0) \\
&=&\bigoplus_{i=0}^{k}\bigoplus_{j=0}^{k-i}V(2k-i,2k-i-j,i+j,i,0,\cdots, 0)\\
&=&
\bigoplus_{i=0}^{k}\bigoplus_{j=0}^{k-i}F(j\pi_1+(2k-2i-2j)\pi_2+j\pi_3+i\pi_4).
\end{eqnarray*}

Note that $k\geqq i+j$. 
\end{lemma}

\begin{lemma}
The lowest weight of $V(2k-i,2k-i-j,j,i,0,\cdots,0)$ is 
\[
-2k+\left(1+\frac{2}{m} \right)i+\left(\frac{1}{2}+\frac{1}{m} \right)j. 
\]
\end{lemma}
\begin{proof}
If $\check w_i$ is the lowest  weight vector in 
$V(2k-i,2k-i-j,i+j,i,0,\cdots,0)$, 
then 
\begin{align*}
Y\cdot \check w_i=&y^{-\frac{1}{2}j-\frac{1}{m}\left\{m(2k-2i-2j)+(m-1)j+(m-2)i \right\}}\check w_i\\
=&y^{-2k+\left(2-\frac{m-2}{m} \right)i+\left(-\frac{1}{2}+2-\frac{m-1}{m} \right)j}\check w_i\\
=&y^{-2k+\left(1+\frac{2}{m} \right)i+\left(\frac{1}{2}+\frac{1}{m} \right)j}\check w_i.
\end{align*}
\end{proof}

Finally, we proceed to detail the 
moduli up to gauge equivalence of maps of holomorphic isometric embeddings of 
${\rm Gr}_m(\mathbf C^{m+2})\to \mathcal Q_p={\rm Gr}_p(W)$  
of degree $k.$  If $V_0=\mathbf C_{k}$ denotes 
the one--dimensional representation $V_{m}(k,k,\cdots,k)$ 
of $\text{U}(m)$, 
then the homogeneous bundle $\text{SU}(m+2)\times_{\mathrm  S(\mathrm U(m)\times \mathrm U(2))} \mathbf C_{-k}$ is 
the complex line bundle $\mathcal O(k) \to \mathbf {\rm Gr}_m(\mathbf C^{m+2})$ of degree $k$. 
Therefore, by the generalisation of do Carmo--Wallach, cf Theorem \ref{GenDW}, $W$ is identified with $H^0({\rm Gr}_m(\mathbf C^{m+2}),\mathcal O(k)).$ 
In order to achieve the description of the moduli Theorems \ref{GenDW} and  \ref{quadmodcpx}    require the specification 
of $\G\mathrm{S}(\mathfrak{m}V_0,V_0)\cap \mathrm  H(W)^\perp $.

\begin{lemma}\label{mV0}
\[\mathfrak m V_0 = \mathbf C^m\otimes \mathbf C^{2*}\otimes \mathbf C_{-k+\frac1m+\frac12}.\]
\end{lemma}

\begin{proof}
Let $\mathbf C^{m}$ denote the standard representation of $\text{U}(m),$ that is, 
$\mathbf C^{m}=V_{m}(1,0,\cdots,0)$. 
The tautological vector bundle $S\to {\rm Gr}_{m}(\mathbf C^{m+2})$ can be identified with the homogeneous bundle ${\rm SU}(m+2)\times_{\mathrm  S(\mathrm U(m)\times \mathrm U(2))}\mathbf C^m.$ Together with the universal quotient bundle $Q\to {\rm Gr}_m\mathbf C^{m+2}$ defined by the exact
sequence $0\to S\to \underline{\mathbf C^{m+2}}\to Q\to 0,$ they define the holomorphic tangent bundle, $T\to {\rm Gr}_m(\mathbf C^{m+2}),\;$ $T=S^*\otimes Q,$ which is  a homogeneous bundle with standard fibre
\[ \left({\mathbf C^m} \otimes \mathbf C_{\frac{1}{m}}\right) ^{\ast} \otimes 
\left(\mathbf C^2 \otimes 
\mathbf C_{-\frac{1}{2}}\right)
={\mathbf C^m}^{\ast} \otimes 
\mathbf C^2 \otimes \mathbf C_{-\frac{1}{m}-\frac{1}{2}}, 
\]
and $T^*\to {\rm Gr}_m(\mathbf C^{m+2})$ where
\[T^{\ast}
={\mathbf C^m}\otimes 
{\mathbf C^2}^{\ast}\otimes \mathbf C_{\frac{1}{m}+\frac{1}{2}}.
\]

Since  $T ^c\cong \mathfrak m^c,$ and $V_0$ is spanned by
the lowest--weight vector 	of $H^0({\rm Gr}_m(\mathbf C^{m+2}),\mathcal O(k)),$ then
\[T^c\otimes V_0=
{\mathbf C^m}\otimes 
{\mathbf C^2}^{\ast}\otimes \mathbf C_{\frac{1}{m}+\frac{1}{2}}
\otimes \mathbf C_{-k}
={\mathbf C^m}\otimes 
{\mathbf C^2}^{\ast}\otimes \mathbf C_{-k+\frac{1}{m}+\frac{1}{2}}.
\]
\end{proof}

\begin{lemma}\label{intersection}
\[{\rm GS}(\mathfrak m V_0,V_0)\cap  \mathrm  H(W)^\perp=V(2k,2k,0,\dots,0).\]
\end{lemma}

\begin{proof} 
The lowest--weight of ${\rm GS}(\mathfrak m V_0,V_0)$ with
respect to $Y$ is
\[
-k-k+\frac{1}{m}+\frac{1}{2}=-2k+\frac{1}{m}+\frac{1}{2}. 
\]
Hence, if $i\geqq 1,$ or $i=0$ and $j>1$, then 
\[
-2k+\left(1+\frac{2}{m} \right)i+\left(\frac{1}{2}+\frac{1}{m} \right)j > 
-2k+\frac{1}{m}+\frac{1}{2}.
\]
It follows that
\begin{eqnarray*}
{\rm GS}(\mathfrak{m}V_0,V_0)
& \subset & 
V(2k,2k,0,\cdots, 0)\qquad \qquad \quad \;\; (i=0, j=0)\\
& &\oplus  \; V(2k,2k-1,1,0,\cdots,0)\qquad (i=0,j=1).
\end{eqnarray*}
On the other hand, since 
\[
V(2k,2k,0,\cdots,0)\subset S^2(F(k\pi_2))
\]
and 
\[
V(2k,2k-1,1,0,\cdots, 0)\subset \wedge^2(F(k\pi_2)),
\] 
we have 
\[
{\rm GS}(\mathfrak{m}V_0,V_0)= V(2k,2k,0,\cdots, 0). 
\]

\end{proof}

Accordingly, applying the Lemma to the decomposition of 
${\rm H}(W)^\perp=S^2  F(k\pi_2)$  
leads to

\begin{cor}\label{intersection 2}
The orthogonal complement to 
$\G\mathrm{S}(\mathfrak{m}V_0,V_0) \oplus \mathbf{R}\,Id$
in $\mathrm{S}(W)$ is 
\[
V_k=
\bigoplus_{i=1}^{2i\leqq k}V(2k-2i,2i,0, \cdots,0).
\]
\end{cor}

\begin{rem}A strightforward computation using hook--length formulae on Young tableaux leads to the explicit, combinatorial expression for the (real) dimension of $V_k$ 
\[\dim V_k = \dim S^2 F(k\pi_2)-\dim V(2k,2k,0,\dots,0)\]
which equals
\begin{eqnarray*}
& & \frac{1}{(k+1)^2}\binom{m-1+k}{m-1}\binom{m-2+k}{m-2}\left(k+1+\binom{m-1+k}{m-1}\binom{m-2+k}{m-2}\right)\\  & & -\frac{2}{2k+1}\binom{m-1+2k}{m-1}\binom{m-2+2k}{m-2}.
\end{eqnarray*}
\end{rem}

\medskip

\begin{thm}\label{gmod1}  
If $f:{\rm Gr}_m(\mathbf C^{m+2}) \to {\rm Gr}_n (\R^{n+2})$ is a full holomorphic isometric embedding of degree $k,$ then 
$n+2 \leq \dim V_k$.

Let $\mathcal M_k$ be the moduli space of full holomorphic isometric embeddings of degree $k$
  of ${\rm Gr}_m(\mathbf C^{m+2})$ into ${\rm Gr}_{N}(\mathbf R^{N+2})$
 by the gauge equivalence of maps, 
where $N+2=\dim V_k. $
Then, $\mathcal M_k$ can be regarded as an open bounded convex body 
in $V_k$.

Let $\overline{\mathcal M_k}$ be the closure of the moduli $\mathcal M_k$ by 
topology induced from 
the inner product. 
Every boundary point of $\overline{\mathcal M_k}$ 
distinguishes a subspace $\mathbf R^{p+2}$ of $\mathbf R^{N+2}$ and 
describes a full holomorphic isometric embedding into ${\rm Gr}_p(\mathbf R^{p+2})$ 
which can be regarded as totally geodesic submanifold of 
${\rm Gr}_{N}(\mathbf R^{N+2})$. 
The inner product on $\mathbf R^{N+2}$ determines 
the orthogonal decomposition of $\mathbf R^{N+2}:$ 
$\mathbf R^{N+2}=\mathbf R^{p+2}\oplus \mathbf R^{{p+2}^{\bot}}$. 
Then the totally geodesic submanifold ${\rm Gr}_p(\mathbf R^{p+2})$ 
can be obtained as the common zero set of sections of 
$Q\to {\rm Gr}_{N}(\mathbf R^{N+2})$,
which belongs to $\mathbf R^{{p+2}^{\bot}}$.  
\end{thm} 

\begin{proof}
The constraint $n\leq N$ is a consequence of (a) in Theorem \ref{GenDW} and Bott--Borel--Weil theorem.

It follows from  (c) 
in Theorem \ref{GenDW}  that $\G\mathrm{S}(\mathfrak{m}V_0,V_0)^\perp$ is a parametrization
of the space of full holomorphic isometric embeddings 
$f: {\rm Gr}_m\mathbf C^{m+2}
\to {\rm Gr}_{N}(\mathbf R^{N+2})$ of degree $k.$ 
Since the standard map into ${\rm Gr}_{N}(\mathbf R^{N+2})$ 
is the composite of the Kodaira embedding
$\mathbf {\rm Gr}_m(\mathbf C^{m+2}) \to \mathbf CP^{\frac{N}{2}}$ and 
the totally geodesic embedding 
$\mathbf CP^{\frac{N}{2}} \hookrightarrow {\rm Gr}_{N}(\mathbf R^{N+2})$,  
we can apply Theorem \ref{quadmodcpx} and Corollary \ref{intersection 2} 
to conclude that $\mathcal M_k$ is 
a bounded connected \emph{open} convex body in  $V_k$ with the
topology induced by the $L^2$ scalar product.

Under the natural compactification in the $L^2$-topology, 
the boundary points correspond to endomorphisms $T$ which are not positive definite,
but positive semi-definite. 
It follows from Theorem \ref{GenDW} that each of these endomorphisms defines 
a full holomorphic isometric embedding 
${\rm Gr}_m(\mathbf C^{m+2}) \to {\rm Gr}_{p}(\mathbf R^{p+2})$,  
of degree $k$ with $p=2k-{\rm dim}\;\mathrm{Ker}\;T,$ whose target embeds in ${\rm Gr}_{N}(\R^{N+2})$ 
as a totally geodesic submanifold. 
The image of the embedding ${\rm Gr}_p(\R^{p+2}) \hookrightarrow {\rm Gr}_{N}(\R^{N+2})$ 
is determined by the common zero set of
sections in  $\mathrm{Ker}\;T$. 
(See also the Remark after Proposition 5.14 in \cite{Na-13} for the geometric meaning of 
the compactification of the moduli space.)  
\end{proof}

\begin{rem} 
It follows 
from Corollary 5.18 in \cite{Na-15} that 
the first condition in (\ref{HDW 2}) is automatically satisfied.  
Alternatively, using the same techniques as in Lemma  \ref{intersection} it can be shown that
\[{\rm GS}(V_0,V_0)\cap {\rm H}(W)^\perp = V(2k,2k,0,\dots,0).\]
\end{rem}

\begin{exam} {\it (Complexified, compactified) Minkowski Space,  ${\rm Gr}_2(\mathbf C^4):$} Let $m=2$ in the previous discussion.  We have that 
\begin{eqnarray*}
\otimes^2 F(k\pi_2) & = & 
\bigoplus_{i=0}^{k}\bigoplus_{j=0}^{k-i}V(2k-i,2k-i-j,i+j,i) \\
&=&\bigoplus_{j=0}^{k}V(2k,2k-j,j,0) \oplus
\bigoplus_{j=0}^{k-1}V(2k-1,2k-j-1,j+1,1) \\
& &\oplus \cdots \oplus  \bigoplus_{j=0}^{1}V(k+1,k-j+1,k-1+j,k-1)\oplus  V(k,k,k,k) 
\end{eqnarray*} 
Regarded as an $\text{SU}(4)$-module,
\[
V(2k-i,2k-i-j,i+j,i)=F(j\pi_1+(2k-2i-2j)\pi_2+j\pi_3),
\quad (k\geqq i+j).\]
Let $\check w_i\in V(2k-2i,2k-i-j,i+j,i)$ be the lowest--weight vector,
then 
\[
Y\cdot \check w_i=y^{-\frac{1}{2}j-\frac{1}{2}\left\{2(2k-2i-2j)+j \right\}}\check w_i
=y^{-2k+2i+j}\check w_i,
\]
so, the lowest weight of $V(2k-2i,2k-2i-j,j,0)$ is 
$-2k+2i+j.$

If $ V_0=\mathbf C_{-k}$ denotes the standard fibre of 
$\mathcal O(k)\to {\rm Gr}_2(\mathbf C^4),$  
Lemma \ref{mV0} shows that
\[\mathfrak m V_0={\mathbf C^2}\otimes 
{\mathbf C^2}^{\ast}\otimes \mathbf C_{\frac{1}{2}+\frac{1}{2}}
\otimes \mathbf C_{-k} 
={\mathbf C^2}\otimes 
{\mathbf C^2}^{\ast}\otimes \mathbf C_{-k+1},\]
therefore $G\mathrm S(\mathfrak{m}V_0,V_0)$ has the weight 
$-k-k+1=-2k+1.$
Hence, if $i\geqq 1$ or $i=0$ and $j>1$, 
\[
-2k+2i+j > -2k+1.
\]
It follows that
\[
{\rm G S}(\mathfrak{m}V_0,V_0)
\subset 
V(2k,2k,0,0)\oplus V(2k,2k-1,1,0)\]
Since 
$V(2k,2k,0,0)\subset S^2(F(k\pi_2))$ 
and 
$V(2k,2k-1,1,0,0)\subset \wedge^2(F(k\pi_2))$, 
we conclude 
\[
{\rm G S}(\mathfrak{m}V_0,V_0)\subset V(2k,2k,0,0). 
\] and therefore the real dimension of $V_k$ is
\[-\frac{2}{3} (1 + k)^2 (1 + 2 k) (3 + 2 k) + 
 \frac{1}{144} (1 + k)^5 (2 + k)^2 (3 + k) (24 + k (4 + k) (7 + k (4 + k))).\]
\hfill $\blacksquare$
\end{exam}

The centraliser $S^1\cong \mathrm U(1)$ of the holonomy subgroup $\mathrm K=\mathrm  S(\mathrm U(m)\times \mathrm U(2))$ of the structure group of the line bundle acts on $\mathcal M_k.$  
For the general theory, see \cite{Na-15}.  Therefore, the same
argument and proof as the one of Theorem 8.1 in \cite{MNT} applies leading to 

\begin{cor}\label{imod}
The  moduli space $\mathbf M_k$ of image equivalence classes of holomorphic isometric embeddings  ${\rm Gr}_m(\mathbf C^{m+2})\to {\rm Gr}_{N}(\mathbf R^{N+2}), N+2=\dim V_k,$ 
of degree $k,$  
is
\[\mathbf M_k=\mathcal M_k\slash S^1.\] 
\end{cor}

\begin{rem}
There is a natural, induced complex structure defined on $\mathcal M_k$ from its embbedding in $ V_k.$ It is also equipped with a compatible metric induced from the inner product, so $\mathcal M_k$ is a K\"ahler manifold.
The aforesaid $S^1$--action preserves the K\"ahler structure on $\mathcal M_k.$ The moment map $\mu:\mathcal M_k\to \mathbf R\;:\; |Id-T^2|^2$ induces the K\"ahler quotient, and 
$\mathbf M_k$ has a foliation   
whose general leaves are the complex projective spaces. 
\end{rem}

\end{document}